\documentclass[12pt, reqno]{amsart}
\usepackage{amsmath, amsthm, amscd, amsfonts, amssymb, graphicx, color}
\usepackage[bookmarksnumbered, colorlinks, plainpages]{hyperref}
\usepackage{amsrefs}
\hypersetup{colorlinks=true,linkcolor=red, anchorcolor=green,
citecolor=cyan, urlcolor=red, filecolor=magenta, pdftoolbar=true}

\newtheorem{theorem}{Theorem}[section]
\newtheorem{lemma}[theorem]{Lemma}

\newtheorem{corollary}[theorem]{Corollary}
\theoremstyle{definition}

\newtheorem{example}[theorem]{Example}

\theoremstyle{remark}
\newtheorem{remark}[theorem]{Remark}
\numberwithin{equation}{section}

\def\rad{\operatorname{rad}}

\def\esup{\operatornamewithlimits{ess\,sup}}

\def\N{\mathbb N}
\def\R{\mathbb R}

\def\ap{\approx}

\def\mf{\mathfrak M}

\def\rn{\R^n}
\def\a{\alpha}

\def\la{\lambda}

\def\vp{\varphi}

\def\M{\mathcal M}

\def\dn{\downarrow}

\def\ls{\lesssim}
\def\gs{\gtrsim}

\def\R{\mathbb R}

\def\M{\mathcal M}

\def\Bxr {{B(x,r)}}

\def\Lploc{L_p^{\rm loc}(\rn)}

\def\Lloc{L_1^{\rm loc}(\rn)}

\begin{document}
\setcounter{page}{1}

\title[A note on boundedness of the maximal function on Morrey spaces]{A note on boundedness of the Hardy-Littlewood maximal operator on Morrey spaces}

\author[A. Gogatishvili and R.Ch. Mustafayev]{A. Gogatishvili$^1$ and R.Ch. Mustafayev$^2$}

\address{$^{1}$ Institute of Mathematics of the Academy of Sciences of the Czech
Republic, \'Zitna~25,  115 67 Prague 1, Czech Republic.}
\email{\textcolor[rgb]{0.00,0.00,0.84}{gogatish@math.cas.cz}}

\address{$^{2}$ Department of Mathematics, Faculty of Science and Arts, Kirikkale
University, 71450 Yahsihan, Kirikkale, Turkey.}
\email{\textcolor[rgb]{0.00,0.00,0.84}{rzamustafayev@gmail.com}}

\thanks{The research of A. Gogatishvili was partly supported by the grants P201-13-14743S
of the Grant Agency of the Czech Republic and RVO: 67985840, by
Shota Rustaveli National Science Foundation grants no. 31/48
(Operators in some function spaces and their applications in Fourier
Analysis) and no. DI/9/5-100/13 (Function spaces, weighted
inequalities for integral operators and problems of summability of
Fourier series). The research of both authors was partly supported
by the joint project between  Academy of Sciences of Czech Republic
and The Scientific and Technological Research Council of Turkey}



\begin{abstract}
In this paper we prove that the Hardy-Littlewood maximal operator is
bounded on Morrey spaces $\mathcal{M}_{1,\lambda}(\rn)$, $0 \le \la
< n$ for radial, decreasing functions on $\rn$.
\end{abstract}

\maketitle

\section{Introduction }

Morrey spaces $\mathcal{M}_{p, \lambda} \equiv \mathcal{M}_{p, \lambda} (\rn)$, were  introduced by C.~Morrey in \cite{M1938} in order to study regularity questions which appear in the Calculus of Variations, and defined as follows:  for $0 \le \lambda \le n$ and $1\le p < \infty$,
$$
\mathcal{M}_{p,\lambda} : = \left\{ f \in   \Lploc:\,\left\| f\right\|_{\mathcal{M}_{p,\lambda }} : =
\sup_{x\in \rn, \; r>0 }
r^{\frac{\lambda-n}{p}} \|f\|_{L_{p}(B(x,r))} <\infty\right\},
$$
where $\Bxr$ is the open ball centered at $x$ of radius $r$.

Note that $\mathcal{M}_{p,0}(\rn) = L_{\infty}(\rn)$ and ${\mathcal M}_{p,n}(\rn) = L_{p}(\rn)$, when $1 \le p < \infty$.

These spaces describe local regularity more precisely than Lebesgue spaces and appeared to be quite useful in the study of the local
behavior of solutions to partial differential equations, a priori estimates and other topics in PDE (cf. \cite{giltrud}).

Given a locally integrable function $f$ on $\rn$ and $0\leq \a <n$,
the fractional maximal function $M_{\a}f$ of $f$ is defined by
$$
M_{\a}f(x):=\sup_{Q \ni x}|Q|^{\frac{\a-n}{n}}\int_Q |f(y)|\,dy,
\qquad (x\in\rn),
$$
where the supremum is taken over all cubes $Q$ containing $x$. The operator $M_{\a}:\,f \rightarrow M_{\a}f$ is called the fractional
maximal operator. $M: = M_{0}$ is the classical Hardy-Littlewood maximal operator.

The study of maximal operators is one of the most important topics in harmonic analysis. These significant non-linear operators, whose behavior are very informative in particular in differentiation theory, provided the understanding and the inspiration for the development of the general class of singular and potential operators
(see, for instance, \cite{stein1970}, \cite{guz1975}, \cite{GR}, \cite{stein1993},
\cite{graf2008}, \cite{graf}).

The boundedness of the Hardy-Littlewood maximal operator $M$ in
Morrey spaces $\mathcal{M}_{p,\lambda}$ was proved by F.~Chiarenza
and M.~Frasca in \cite{ChiFra1987}: It was shown  that $Mf$ is a.e. finite if $f \in \mathcal{M}_{p,\lambda}$ and an estimate
\begin{equation}\label{ChiFr}
\|Mf\|_{\mathcal{M}_{p,\lambda}} \le c \|f\|_{\mathcal{M}_{p,\lambda}}
\end{equation}
holds if $1 < p < \infty$ and $0 < \lambda < n$, and a weak type estimate \eqref{ChiFr} replaces for $p = 1$, that is, the inequality
\begin{equation}\label{ChiFrWeakType}
t |\{Mf > t\} \cap B(x,r)| \le cr^{n-\lambda} \|f\|_{\mathcal{M}_{1,\lambda}}
\end{equation}
holds with constant $c$ independent of $x,\,r,\,t$ and $f$.

In this paper we show that \eqref{ChiFr} is not true for $p = 1$. According to our example the right result is  \eqref{ChiFrWeakType}. If restricted to the cone of radial, decreasing functions on $\rn$, inequality \eqref{ChiFr} holds true for $p = 1$.

The paper is organized as follows. We start with notation and
preliminary results in Section~\ref{s.2}.
In Section \ref{s.3}, we prove that the Hardy-Littlewood maximal operator $M$ is bounded on $\M_{1,\la}$, $0 < \la < n$, for radial, decreasing functions, and we give an example which shows that $M$ is not bounded on $\M_{1,\la}$, $0 < \la < n$.

\

\section{Notations and Preliminaries}\label{s.2}

\

Now we make some conventions. Throughout the paper, we always denote
by $c$ a positive constant, which is independent of main parameters, but it may vary from line to line. By $a\ls b$ we mean that $a\le c b$ with some positive
constant $c$ independent of appropriate quantities. If $a\ls b$ and
$b\ls a$, we write $a\approx b$ and say that $a$ and $b$ are
equivalent. For a measurable set $E$, $\chi_E$ denotes the characteristic function of $E$.

Let $\Omega$ be any measurable subset of $\rn$, $n\geq 1$. Let
$\mf(\Omega)$ denote the set of all measurable functions on $\Omega$
and $\mf_0 (\Omega)$ the class of functions in $\mf (\Omega)$ that
are finite a.e., while $\mf^{\dn} (0,\infty)$ ($\mf^{+,\dn} (0,\infty)$) is used to denote the subset of those functions which are non-increasing (non-increasing and non-negative) on $(0,\infty)$. Denote by $\mf^{\rad,\dn} = \mf^{\rad,\dn}(\rn)$
the set of all measurable, radial, decreasing functions on $\rn$, that is,
$$
\mf^{\rad,\dn} : = \{f \in \mf(\rn):\, f(x) = \vp(|x|),\,x \in \rn
\,\mbox{with}\,\vp \in \mf^{\dn}(0,\infty)\}.
$$

Recall that $Mf \ap Hf$, $f \in \mf^{\rad,\dn}$, where
$$
Hf (x) : = \frac{1}{|B(0,|x|)} \int_{B(0,|x|)} |f(y)|\,dy
$$
is $n$-dimensional Hardy operator. Obviously, $Hf \in  \mf^{\rad,\dn}$, when $f \in \mf^{\rad,\dn}$.

For $p\in (0,\infty]$, we define the functional
$\|\cdot\|_{p,\Omega}$ on $\mf(\Omega)$ by

\begin{equation*}
\|f\|_{p,\Omega}:=
\begin{cases}
    (\int_{\Omega} |f(x)|^p \,dx)^{1/p}  &\text{if } \ \ p<\infty,\\
    \esup_{\Omega} |f(x)|              &\text{if } \ \ p=\infty.
\end{cases}
\end{equation*}

The Lebesgue
space $L_p(\Omega)$ is given by
\begin{equation*}
L_p(\Omega):= \{f\in \mf(\Omega): \|f\|_{p,\Omega}<\infty\}
\end{equation*}
and it is equipped with the quasi-norm $\|\cdot\|_{p,\Omega}$.

The decreasing rearrangement (see, e.g., \cite[p. 39]{BS}) of a
function $f \in \mf_0 (\rn)$  is defined by
$$
f^*(t) : = \inf\left\{\la >0 : |\{x\in\rn: |f(x)|>\la \}| \leq t
\right\}\quad (0<t<\infty).
$$

\section{Boundedness of $M$ on $\M_{1,\la}$ for radial, decreasing functions}\label{s.3}

Reacall that
\begin{align*}
M_{\a}f (x) & \ap \sup_{B \ni x}   |B|^{\frac{\a-n}{n}}\int_B |f(y)|\,dy \\
& \ap \sup_{r > 0}   |B(x,r)|^{\frac{\a-n}{n}}\int_{B(x,r)} |f(y)|\,dy,
~ (x\in\rn),
\end{align*}
where the supremum is taken over all balls $B$ containing $x$.

In order to prove our main result we need the following auxiliary lemmas.
\begin{lemma}\label{lem999.1}
Assume that $0 < \la <  n$. Let  $f \in \mf^{\rad,\dn}(\rn)$ with $f(x) = \vp (|x|)$. The equivalency
$$
\|f\|_{\M_{1,\la}} \ap \sup_{x>0} x^{\la - n} \int_0^x
|\vp (\rho)|\rho^{n-1}\,d\rho
$$
holds with positive constants independent of $f$.
\end{lemma}
\begin{proof}
Recall that
$$
\|f\|_{\M_{1,\la}} \ap
\sup_{B}|B|^{\frac{\la-n}{n}}\int_{B} f = \| M_{\la}f\|_{\infty}, ~ f \in \mf(\rn).
$$
Switching to polar coordinates, we have that
\begin{align*}
M_{\la}(f)(y) & \gs |B(0,|y|)|^{\frac{\la-n}{n}} \int_{B(0,|y|)} |f(z)|\,dz \\
& \ap |y|^{\la - n}\int_0^{|y|}
|\vp(\rho)|\rho^{n-1}\,d\rho.
\end{align*}
Consequently,
\begin{align*}
\|f\|_{\M_{1,\la}} & \gs \esup_{y \in \rn} |y|^{\la - n}\int_0^{|y|}
|\vp(\rho)|\rho^{n-1}\,d\rho  \\
& = \sup_{x > 0} x^{\la - n}\int_0^x
|\vp(\rho)|\rho^{n-1}\,d\rho,
\end{align*}
where $f (\cdot) = \vp (|\cdot|)$.

On the other hand,
\begin{align*}
\|f\|_{\M_{1,\la}} & \ls \sup_{B}|B|^{\frac{\la-n}{n}}\int_0^{|B|} f^* (t)\,dt \\
& = \sup_{B}|B|^{\frac{\la-n}{n}}\int_0^{|B|}  |\vp(t^{\frac{1}{n}})|\,dt \\
& \ap \sup_{B}|B|^{\frac{\la-n}{n}}\int_0^{|B|^{\frac{1}{n}}}  |\vp(\rho)|\rho^{n-1}\,d\rho \\
& = \sup_{x > 0} x^{\la - n}\int_0^x
|\vp(\rho)|\rho^{n-1}\,d\rho,
\end{align*}
where  $f (\cdot) = \vp (|\cdot|)$.
\end{proof}

\begin{corollary}\label{cor999.1}
Assume that $0 < \la <  n$. Let  $f \in \mf^{\rad,\dn}(\rn)$ with $f(x) = \vp (|x|)$. The equivalency
$$
\|Mf\|_{\M_{1,\la}} \ap \sup_{x>0} x^{\la - n} \int_0^x
\vp (\rho)\rho^{n-1} \ln \left(\frac{x}{\rho}\right)\,d\rho
$$
holds with positive constants independent of $f$.
\end{corollary}
\begin{proof}
Let $f \in \mf^{\rad,\dn}$  with $f(x) = \vp (|x|)$. Since $Mf \ap Hf$ and $Hf \in \mf^{\rad,\dn}$, by Lemma \ref{lem999.1}, switching to polar coordinates, using Fubini's Theorem,  we have that
\begin{align*}
\|Mf\|_{\M_{1,\la}} & \ap \sup_{x>0} x^{\la - n} \int_0^x
\left(\frac{1}{|B(0,t|} \int_{B(0,t)} |f(y)|\,dy \right)t^{n-1}\,dt \\
& \ap \sup_{x>0} x^{\la - n} \int_0^x
\frac{1}{t} \int_0^t
\vp (\rho)\rho^{n-1}\,d\rho\,dt \\
& = \sup_{x>0} x^{\la - n} \int_0^x
\vp (\rho)\rho^{n-1} \ln \left(\frac{x}{\rho}\right)\,d\rho.
\end{align*}
\end{proof}

\begin{lemma}\label{lem999.2}
    Assume that $0 < \la <  n$. Let  $f \in \mf^{\rad,\dn}$ with $f(x) = \vp (|x|)$. The inequality
$$
    \|Mf\|_{\M_{1,\la}} \ls \|f\|_{\M_{1,\la}},~ f \in \mf^{\rad,\dn}
$$
holds if and only if the inequality
$$
 \sup_{x>0} x^{\la - n} \int_0^x
 \vp (\rho)\rho^{n-1} \ln \left(\frac{x}{\rho}\right)\,d\rho \ls \sup_{x>0} x^{\la - n} \int_0^x
\vp (\rho)\rho^{n-1}\,d\rho,~ \vp \in \mf^{+,\dn}(0,\infty)
$$
holds.
\end{lemma}
\begin{proof}
    The statement immediately follows from Lemma \ref{lem999.1} and Corollary \ref{cor999.1}.
\end{proof}

\begin{lemma}\label{lem999.5}
Let $0 <\la < n$. Then inequality
\begin{equation}\label{eq.999}
 \sup_{x>0} x^{\la - n} \int_0^x
 \vp (\rho)\rho^{n-1} \ln \left(\frac{x}{\rho}\right)\,d\rho \ls \sup_{x>0} x^{\la - n} \int_0^x
\vp (\rho)\rho^{n-1}\,d\rho
\end{equation}
holds for all $\vp \in \mf^{+,\dn}(0,\infty)$.
\end{lemma}
\begin{proof}
Indeed:
\begin{align*}
\sup_{x>0} x^{\la - n} \int_0^x
\vp (\rho)\rho^{n-1} \ln \left(\frac{x}{\rho}\right)\,d\rho & \\
& \hspace{-3cm} = \sup_{x>0} x^{\la - n} \int_0^x
\frac{1}{t} \int_0^t
\vp (\rho)\rho^{n-1}\,d\rho\,dt & \\
& \hspace{-3cm} = \sup_{x>0} x^{\la - n} \int_0^x
t^{n - \la - 1} t^{\la - n}\int_0^t
\vp (\rho)\rho^{n-1}\,d\rho\,dt \\
& \hspace{-3cm} \le  \sup_{t > 0}t^{\la - n}\int_0^t \vp (\rho)\rho^{n-1}\,d\rho \cdot \left(  \sup_{x>0} x^{\la - n} \int_0^x
t^{n - \la - 1} \,dt \right) \\
& \hspace{-3cm} \ap \sup_{t > 0}t^{\la - n}\int_0^t
\vp (\rho)\rho^{n-1}\,d\rho.
\end{align*}
\end{proof}

Now we are in position to prove our main result.
\begin{theorem}\label{main.thm}
Assume that $0 < \la <  n$. The inequality
\begin{equation}\label{main}
\|Mf\|_{\M_{1,\la}} \ls \|f\|_{\M_{1,\la}}
\end{equation}
holds for all $f \in \mf^{\rad,\dn}$ with constant independent of $f$.
\end{theorem}
\begin{proof}
The statement follows by Lemmas \ref{lem999.2} and \ref{lem999.5}.
\end{proof}

\begin{remark}
Note that inequality \eqref{main} holds true when $\la = 0$, for  $\mathcal{M}_{1,0}(\rn) = L_{\infty}(\rn)$ and $M$ is bounded on $L_{\infty}(\rn)$.
\end{remark}

\begin{remark}
It is obvious that the statement of Theorem \ref{main.thm} does not hold when $\la = n$, for in this case ${\mathcal M}_{1,n}(\rn) = L_{1}(\rn)$ and the inequality
$$
\|Mf\|_{L_1(\rn)} \ls \|f\|_{L_1(\rn)}
$$
is true only for $f = 0$ a.e., which follows from the fact that $Mf (x) \ap |x|^{-n}$ for $|x|$ large when $f \in \Lloc$.
\end{remark}

\begin{example}
We show that $M$ is not bounded on ${\mathcal M}_{1,\la}(\rn)$, $0 <
\la < n$. For simplicity let $n=1$ and $\la = 1 / 2$. Consider the
function
$$
f(x) = \sum_{k=0}^{\infty} \chi_{[k^2,k^2 + 1]}(x).
$$
Then
\begin{align*}
\|f\|_{{\mathcal M}_{1,1/2}(\R)} = \sup_{I}|I|^{-1/2} \int_I f \le \sup_{I:\,|I| \le 1}|I|^{-1/2} \int_I f + \sup_{I:\,|I| > 1}|I|^{-1/2} \int_I f,
\end{align*}
where the supremum is taken over all open intervals $I \subset \R$.
It is easy to see that
$$
\sup_{I:\,|I| \le 1}|I|^{-1/2} \int_I f \le \sup_{I:\,|I| \le 1}|I|^{1/2} \le 1.
$$
Note that
\begin{align*}
\sup_{I:\,|I| > 1}|I|^{-1/2} \int_I f & = \sup_{m \in \N} \sup_{I:\,m < |I| \le m+1}|I|^{-1/2} \int_I f \\
& = \sup_{m \in \N} \sup_{I:\,m < |I| \le m+1}|I|^{-1/2} \int_I \left( \sum_{k=0}^{\infty} \chi_{[k^2,k^2 + 1]}(x)\right)\,dx \\
& = \sup_{m \in \N} \sup_{I:\,m < |I| \le m+1}|I|^{-1/2} \left| I \cap \bigcup_{k=0}^{\infty} [k^2,k^2 + 1]\right|.
\end{align*}
Since
$$
\left| I \cap \bigcup_{k=0}^{\infty} [k^2,k^2 + 1]\right| \le \left| [0,m+1] \cap \bigcup_{k=0}^{\infty} [k^2,k^2 + 1]\right|
$$
for any interval $I$ such that $m < |I| \le m + 1$, we obtain that
\begin{align*}
\sup_{I:\,|I| > 1}|I|^{-1/2} \int_I f
& \ls \sup_{m \in \N} m^{-1/2} \left| [0,m+1] \cap \bigcup_{k=0}^{\infty} [k^2,k^2 + 1]\right| \\
& \ls \sup_{m \in \N} m^{-1/2} m^{1/2} = 1.
\end{align*}
Consequently, we arrive at
$$
\|f\|_{{\mathcal M}_{1,1/2}(\R)} \ls 2.
$$
On the other hand, since
\begin{align*}
M f \ge \sum_{k=0}^{\infty} \left( \chi_{[k^2,k^2 + 1]} + \frac{1}{x - k^2} \chi_{[k^2 + 1,k^2 + k + 1]} + \frac{1}{(k+1)^2 + 1 - x} \chi_{[k^2 + k + 1, (k+1)^2]} \right),
\end{align*}
we have that
\begin{align*}
\|Mf\|_{{\mathcal M}_{1,1/2}(\R)} \ge \sup_{k \in \N} k^{-1} \int_0^{k^2} Mf & \ge \sup_{k \in \N} k^{-1} \sum_{i=1}^{k-1}\int_{i^2}^{(i+1)^2}  Mf \\
& \ge \sup_{k \in \N} k^{-1}\sum_{j=1}^{k-1} \ln j \gs \sup_{k \in \N} \ln k = \infty.
\end{align*}
\end{example}

\end{document}